\documentclass[article]{amsart}
\usepackage{cases}
\usepackage{amsmath, amsfonts, pifont, amssymb}
\usepackage{verbatim}
\usepackage[bookmarks=true]{hyperref}
\usepackage{geometry}
\newtheorem{theorem}{Theorem}[section]
\newtheorem{lemma}[theorem]{Lemma}
\newtheorem{proposition}[theorem]{Proposition}

\newcommand{\beq}{\begin{equation}}
\newcommand{\eeq}{\end{equation}}
\newcommand{\beqq}{\begin{equation*}}
\newcommand{\eeqq}{\end{equation*}}

\theoremstyle{definition}

\theoremstyle{remark}
\newtheorem{remark}[theorem]{Remark}

\numberwithin{equation}{section}

\numberwithin{equation}{section}

\begin{document}

\address{Chenjie Fan
\newline \indent Academy of Mathematics and Systems Science, CAS, China\indent }
\email{fancj@amss.ac.cn}

\address{Zehua Zhao
\newline \indent Department of Mathematics and Statistics, Beijing Institute of Technology,
\newline \indent MIIT Key Laboratory of Mathematical Theory and Computation in Information Security,
\newline \indent  Beijing, China. \indent}
\email{zzh@bit.edu.cn}

\title[]{A note on decay property of nonlinear Schr\"odinger equations}
\author{Chenjie Fan and Zehua Zhao}
\maketitle

\setcounter{tocdepth}{1}
\tableofcontents

\begin{abstract}
In this note, we show the existence of a special solution $u$ to defocusing cubic  NLS in $3d$, which lives in $H^{s}$ for all $s>0$, but scatters to a linear solution in a very slow way. We prove for this $u$, for all $\epsilon>0$, one has $\sup_{t>0}t^{\epsilon}\|u(t)-e^{it\Delta}u^{+}\|_{\dot{H}^{1/2}}=\infty$. Note that such a slow asymptotic convergence is impossible if one further pose the initial data of $u(0)$ be in $L^{1}$. We expect that similar construction hold the for other NLS models. It can been seen the slow convergence is caused by the fact that there are delayed backward scattering profile in the initial data, we also illustrate why $L^{1}$ condition of initial data will get rid of this phenomena.
\end{abstract}
\bigskip

\bigskip


\section{Introduction}
\subsection{Statement of main results}
Defocusing NLS (nonlinear Schr\"odinger equations) with power-type nonlinearity in $d$-dimensional Euclidean space reads as follows
\begin{equation}\label{maingeneral}
    \left(i\partial_t+ \Delta_{\mathbb{R}^{d}} \right) u=   |u|^{p} u \quad u(0,x)=u_0(x).
\end{equation}
There has been extensive research activities on those models in the area of dispersive PDEs. If one considers proper choice of $(p,d)$ and, the scattering behavior could be obtained for many models (see for example \cite{bourgain1999global,colliander2008global,Dodson3,kenig2006global,tao2006nonlinear} and reference therein), i.e. the solution $u$ is global and for some $u_\pm\in \dot{H}^{s}$,
\begin{equation}
\|u(t)-e^{it\Delta}u^{\pm}\|_{\dot{H}_{x}^{s}}\rightarrow 0, \text{ as } t\rightarrow \pm \infty,
\end{equation}
This indicates that the dynamics of the nonlinear solution $u$ resembles (converges to) linear solution in the long time. A natural question is: can one describe the scattering behavior in a more \textbf{quantitative} way? For example, to describe the convergence rate of the nonlinear solution, that is, finding a continuous, decreasing function of time $f(t)$ such that, for $t>0$
\begin{equation}
  \|u(t)-e^{it\Delta}u^{+}\|_{\dot{H}_{x}^{s}}  \lesssim_{data} f(t).
\end{equation}
(Here the notation $\lesssim_{data}$ means the implicitly constant in the inequality depends on the initial data.) Another aspect is considering scattering rate, i.e., the decay of the scattering norm, for $s>0$,
\begin{equation}
  \|u(t)\|_{L_{t,x}^{\frac{2(d+2)}{d-2s_c}}(t\geq s)}    \lesssim_{data} f(s).
\end{equation}
Here the spacetime norm $L_{t,x}^{\frac{2(d+2)}{d-2s_c}}$ is the scattering norm for NLS \eqref{maingeneral}, which is an important norm for NLS models. For example, the finiteness of scattering implies scattering behavior (see \cite{tao2006nonlinear} for more explanations).
 
 The above two estimates describes how fast the nonlinear solution of NLS scatters to linear solutions asymtotically. We conjecture that, for NLS initial value problem \eqref{maingeneral}, assuming the initial data $u_0$ in $H^s \cap L^1$ (for $s>0$ large enough),
\begin{equation}\label{eq: convergencerate}
  \|u(t)-e^{it\Delta}u^{+}\|_{\dot{H}_{x}^{s}}  \lesssim_{data} t^{-\alpha_1},
\end{equation}
and
\begin{equation}\label{eq: scatteringrate}
  \|u(t)\|_{L_{t,x}^{\frac{2(d+2)}{d-2s_c}}(t\geq s)}    \lesssim_{data} t^{-\alpha_2},
\end{equation}
for some $\alpha_1,\alpha_2>0$ which depends on the specific model.

In this note, we focus on a typical model of this type, i.e. 3d  cubic NLS,
\begin{equation}\label{maineq}
    \left(i\partial_t+ \Delta_{\mathbb{R}^{3}} \right) u=   |u|^{2} u \quad u(0,x)=u_0(x).
\end{equation}
See \cite{cazenave2003semilinear,kenig2010scattering,tao2006nonlinear} for more background regarding this model. We will prove the \textbf{necessity} of $L^1$ assumption for initial data $u_{0}$, for \eqref{eq: scatteringrate} and \eqref{eq: convergencerate}. 

The main result of this note reads
\begin{theorem}\label{mainthm}
Consider initial value problem \eqref{maineq}. There exists $u_0 \in H^{100}$ such that, for any $\epsilon>0$, 
\begin{equation}\label{eq1}
    \sup_{t>0} t^{\epsilon}\|u(s,x)\|_{L^{5}_{s,x}(t,+\infty)}=\infty.
\end{equation}
and
\begin{equation}\label{eq2}
    \sup_{t>0} t^{\epsilon}\|u(t)-e^{it\Delta}u^{+}\|_{\dot{H}_{x}^{\frac{1}{2}}}=\infty,
\end{equation}
where $u^{+}$ is the final state where the solution $u$ scatters to when $t \rightarrow \infty$.
\end{theorem}

The construction of $u_{0}$ is quite explicit. 

\begin{equation}\label{construct}
   u_0:=\sum_{n \geq 0}\frac{1}{n^2}e^{-i10^{n}\Delta}\phi.
\end{equation}
Here $\phi$ is a Schwarz function. We note that it is enough to consider $\phi \in H^{10}\cap L^1$ for example. We consider Schwarz function for convenience.
\begin{remark}
We note that such construction is inspired by concentration compactness method. See \cite{kenig2006global,kenig2008global,keraani2001defect} and the reference therein.
\end{remark}
We will fix $\phi$ and this $u_{0}$.

Clearly, since $u_0 \in H^1$, the corresponding solution $u$ scatters to some $e^{it\Delta} u^{+}$ as $t \rightarrow \infty$ (see \cite{colliander2004global,kenig2010scattering}).

\begin{remark}
By slightly modifying the proof of Theorem \ref{mainthm}, one can also prove that, for any given $g$ s.t $g(t)$ monotonically increases to $\infty$ as $t$ goes to infinity, one may construct $u$ lives in $H^{s}$, $s$ large, so that 
\begin{equation}
    \sup_{t>0} g(t)\|u(s,x)\|_{L^{5}_{s,x}(t,+\infty)}=\infty,
\end{equation}
and
\begin{equation}
    \sup_{t>0} g(t)\|u(t)-e^{it\Delta}u^{+}\|_{\dot{H}_{x}^{\frac{1}{2}}}=\infty.
\end{equation}

Basically, the sequence $\{10^{n}\}_{n}$ are quite separated (compared to $n^{2}$) as $n$ goes to infinity. And one may make it as separate as one want.For example, one replaces $10^{n}$ by $a_{n}$, and imposes further $\frac{1}{n^{100}}g(a_{n})$ goes to infinity.
\end{remark}

In the appendix, we also present  the proof of scattering rate and the convergence rate when the data is in $H^4 \cap L^1 (\mathbb{R}^3)$ space based on our recent work \cite{fan2021decay}.
  
We note that analogous results of this article can be generalized to other NLS models in a natural way.

Note that, since all $H^{1}$ initial data $u_{0}$ to \eqref{maineq} will give a global solution $u$, with $\|u\|_{L_{t,x}^{5}}\lesssim 1$. Thus, for all $\delta>0$, there exists a $L$, such that 
\begin{equation}\label{eq: preone}
\|u\|_{L^{5}_{t,x}([L,\infty)\times \mathbb{R}^3)}\leq \delta.	
\end{equation}
This $L$ does not depend on size of $\|u_{0}\|_{H^{1}}$, but also the profile (or shape) of $u_{0}$. And this cannot be improved, essentially due to time translation symmetry of the equation, \eqref{maineq}.

However, for $u_{0}\in H^{1}\cap L^{1}$, we have
\begin{theorem}\label{thm: onemore}
Let $u$ solves \eqref{maineq}, with $u_{0}\in H^{1}\cap L^{1}$. For all $\delta>0$, there exists $L>0$, depending only on $\|u_{0}\|_{H^{1}\cap L^{1}}$, so that 
\begin{equation}
	\|u\|_{L_{t,x}^{5}([L,\infty)\times \mathbb{R}^3)}\leq \delta.
\end{equation}
\end{theorem}

One application of Theorem \ref{thm: onemore} is the following.
It has been proved in \cite{fan2021decay}, see also (\cite{guodecay}) that for initial data $u_{0}\in H^{4}\cap L^{1}$, the associated nonlinear solution $u$ to \eqref{maineq} satisfies
\begin{equation}\label{eq: od}
	\|u\|_{L_{x}^{\infty}}\lesssim C_{u_{0}}t^{-3/2}.
\end{equation}
 
 Note that this $C_{u_{0}}$ depending on the profile of $u_{0}$ is that we were arguing with \eqref{eq: preone} rather than Theorem \ref{thm: onemore}.
 
 Thus, with Theorem \ref{thm: onemore}, one enhance \eqref{eq: od} into
 \begin{equation}
 	\|u\|_{L_{x}^{\infty}}\lesssim C_{\|u_{0}\|_{H^{4}\cap L^{1}}}t^{-3/2}.
 \end{equation}
 
 It remains an interesting problem to further characterize $C_{\|u_{0}\|_{H^{4}\cap L^{1}}}$.

Now we summarize the global well-posedness and scattering result for \eqref{maineq}. As a corollary of low regularity results \cite{colliander2004global}, \cite{kenig2010scattering}, one has 
\begin{proposition}\label{global}
Initial value problem \eqref{maineq} is globally well-posed and scattering in $H^1$ space. More precisely, for any $u_0$ with finite energy, $u_{0}\in H^{1}$, there exists a
unique global solution $u\in C^0_t({H}^1_x)\cap L^{5}_{x,t}$ such that
\begin{equation}
    \int_{-\infty}^{\infty} \int_{\mathbb{R}^3}|u(t,x)|^{5}dxdt\leq C_{\|u_{0}\|_{H^{1}}},
\end{equation}
for some constant $C(\|u_{0}\|_{H^{1}})$ that depends only on $\|u_{0}\|_{H^{1}}$.
And if $u_0 \in H^s$ for some $s>1$, then $u(t)\in H^s$ for all time $t$, and one has the uniform bounds
\begin{equation}
    \sup\limits_{t\in \mathbb{R}} \|u(t)\|_{H^s} \leq C_{\|u_{0}\|_{H^{1}}}\|u_0\|_{H^s}.
\end{equation}
\end{proposition}
Thus, by the above proposition, one may assume 
\begin{equation}\label{eq: boundglobal}
   ||u(t)||_{L^{\infty}_tH_x^4} \leq M_1,
\end{equation}
where $u$ solves \eqref{maineq} with smooth enough data (at least $H^4$).
\subsection{Notations}
Throughout this note, we use $C$ to denote  the universal constant and $C$ may change line by line. We say $A\lesssim B$, if $A\leq CB$. We say $A\sim B$ if $A\lesssim B$ and $B\lesssim A$. We also use notation $C_{B}$ to denote a constant depends on $B$. We use usual $L^{p}$ spaces and Sobolev spaces $H^{s}$.
\subsection{Acknowledgment}
We thank Zihua Guo for helpful comments. Fan was partially supported in National Key R\&D Program of China, 2021YFA1000800, and NSFC grant No.11688101. Zhao was partially supported by the NSF grant of China (No. 12101046) and the Beijing Institute
of Technology Research Fund Program for Young Scholars.
\section{Proof of Theorem 1.1}
In this section, we prove Theorem \ref{mainthm}, assuming several decay estimates whose proofs will be included in the next section. 
We introduce some notations as below.
\begin{equation}
    u_{0,\leq N}:=\sum_{0 \leq n \leq N}\frac{1}{n^2}e^{-i10^{n}\Delta}\phi,
\end{equation}
\begin{equation}
    u_{0, N}:=u_{0,\leq N}-u_{0,\leq N-1},
\end{equation}
and
\begin{equation}
    u_{0,> N}:=\sum_{ n > N}\frac{1}{n^2}e^{-i10^{n}\Delta}\phi.
\end{equation}
We then define $u_{\leq N},u_{N},u_{>N}$ to be the associated nonlinear solutions respectively. 

We want to confirm that $u_{N+1}$ dominates the dynamics of $u$ within $t\in [10^{N+1},10^{N+1}+1]$. 

In the rest of this section, all implicit constants may depend on $\phi$, but are otherwise universal, and $\alpha_1,\alpha_2,...$ are fixed positive constants and can be written out explicitly if necessary.

We first present several decay estimates here. 
\begin{lemma}\label{lem: decay}
For all $t\geq 2\times 10^N$, we have 
\begin{equation}\label{decay1}
\|u_{\leq N}(t)\|_{L^{\infty}_x} \lesssim (t-10^N)^{-\frac{3}{2}}. 
\end{equation}
For all $t \leq 5 \times 10^N$,
\begin{equation}\label{deca2}
  \|u_{N+1}(t)\|_{L^{\infty}_x} \lesssim (10^{N+1}-t)^{-\frac{3}{2}}.   
\end{equation}
For all $0\leq t \leq 2 \times 10^{N+1}$
\begin{equation}\label{decay3}
   \|u_{>N+1}(t)\|_{L^{\infty}_x} \lesssim (10^{N+2}-t)^{-\frac{3}{2}}.   
\end{equation}
\end{lemma}
We will prove Lemma \ref{lem: decay} in the next section.
Based on the above decay estimates, doing integration, correspondingly, we can directly obtain 
\begin{lemma}\label{property}
For $t\geq 2\times 10^N$, 
\begin{equation}\label{property1}
\|u_{\leq N}(t)\|_{L^5_{t,x}[t,\infty]}\lesssim (t-10^N)^{-\alpha_1},
\end{equation}
and for all $t\in [10^{N+1}-10,10^{N+1}+10]$, for all $s>0$
\begin{equation}\label{property2}
    \|e^{is\Delta}u_{\leq N}(t)-u_{\leq N}(t+s)\|_{\dot{H}^{\frac{1}{2}}} \lesssim (10^N)^{-\alpha_2}.
\end{equation}
For all $t \leq 5 \times 10^N$,
\begin{equation}\label{property3}
    \|u_{N+1}(t)\|_{L^5_{t,x}} \lesssim (10^{N+1})^{-\alpha_3}.
\end{equation}
For all $0<t \leq 2\times 10^{N+1}$,
\begin{equation}\label{property4}
\|u_{>N+1}\|_{L^{5}_{t,x}[0,2\times 10^{N+1}]} \lesssim (10^N)^{-\alpha_4}   
\end{equation}
and for all $t\in [10^{N+1}-10,10^{N+1}+10]$, for all $s>0$,
\begin{equation}\label{property5}
    \|e^{is\Delta}u_{>N+1}(t)-u_{>N+1}(t+s)\|_{\dot{H}^{\frac{1}{2}}} \lesssim (10^N)^{-\alpha_5}.    
\end{equation}
\end{lemma}
\begin{proof}
Obviously, \eqref{property1}, \eqref{property3} and \eqref{property4} follows from \eqref{decay1}, \eqref{deca2} and \eqref{decay3} respectively by doing integration directly.

We turn on the proof of  \eqref{property2} and \eqref{property5}. They are similar so we focus on \eqref{property2}. By Duhamel formula, we write
\begin{equation}
 u_{\leq N}(t+s)-e^{is\Delta}u_{\leq N}(t)=i\int_t^{t+s} e^{i(t+s-\tau)\Delta} |u_{\leq N}(\tau)|^2u_{\leq N}(\tau) d\tau   
\end{equation}
then it follows from \eqref{property1} by Strichartz estimate.
\end{proof}
 Assuming \eqref{eq2} does not hold, there exists $\epsilon_0>0$ such that
\begin{equation}
  \|u(t)-e^{it\Delta}u^{+}\|_{\dot{H}_{x}^{\frac{1}{2}}} \lesssim t^{-\epsilon_0}.
\end{equation}
We note that 
\begin{equation}
  \|e^{-it\Delta}u(t)-u^{+}\|_{\dot{H}_{x}^{\frac{1}{2}}} \lesssim t^{-\epsilon_0}
\end{equation}
and
\begin{equation}
  \|e^{-i(t-s)\Delta}u(t-s)-u^{+}\|_{\dot{H}_{x}^{\frac{1}{2}}} \lesssim t^{-\epsilon_0}.
\end{equation}
Thus, considering the difference, we have 
\begin{equation}
  \|e^{-i(t-s)\Delta}u(t-s)-e^{-it\Delta}u(t)\|_{\dot{H}_{x}^{\frac{1}{2}}} \lesssim t^{-\epsilon_0},
\end{equation}
which implies \eqref{property2}.

Now we consider the dynamics of $u_{N+1}$ at interval $[10^{N+1},10^{N+1}+1]$. 

We recall the global result (Theorem \ref{global}) for \eqref{maineq}. Noticing the size of initial data, we have
\begin{equation}
    \|u_{N+1}\|_{L^5_{t,x}} \lesssim \frac{1}{N^2}.
\end{equation}
Running the standard persistence of regularity argument (i.e. using the finite scattering norm to control the Strichartz norm, see Section 3 in \cite{colliander2008global} for example), we obtain
\begin{equation}\label{Stri}
    \|u_{N+1}\|_{\dot{S}^{\frac{1}{2}}} \lesssim \frac{1}{N^2}.
\end{equation}
Recall Duhamel formula, for $t\in [0,1]$, we write $u_{N+1}(t)$ as
\begin{equation}
   u_{N+1}(t+10^{N+1})=e^{it\Delta}u_{N+1}(10^{N+1})+i\int_{10^{N+1}}^{t+10^{N+1}} e^{i(t+10^{N+1}-s)\Delta} (u_{N+1}|u_{N+1}|^2)(s) ds.
\end{equation}
Here, by Duhamel formula, we note that we write $u_{N+1}(10^{N+1})$ as
\begin{equation}
u_{N+1}(10^{N+1})=\frac{1}{(N+1)^2}\phi+i\int_0^{10^{N+1}} e^{i(10^{N+1}-s)\Delta} (u_{N+1}|u_{N+1}|^2)(s) ds.    
\end{equation}
For the Duhamel term, we can estimate it using \eqref{Stri}, which shows that this term is a lower order term. Thus we see $u_{N+1}(10^{N+1})$ is at the level of $\frac{1}{N^2}$. 

Moreover, by further Duhamel formula expansion, we have
\begin{equation}
   u_{N+1}(t+10^{N+1})=e^{it\Delta}u_{N+1}(10^{N+1})+i\frac{1}{(N+1)^6} \int_0^t e^{i(t-s)\Delta} (e^{is\Delta}\phi|e^{is\Delta}\phi|^2)(s) ds+\mathcal{O}(\frac{1}{N^8}),
\end{equation}
where the last term can well controlled using the known global bounds (Lemma \ref{property}) and it be regarded as error term which is small when $N$ is taken large. 

We note that there at least exists $0<\tau_0<1$ such that 
\begin{equation}
  \| \int_0^{\tau_0} e^{i(\tau_0-s)\Delta} (e^{is}\phi|e^{is}\phi|^2)(s) ds \|_{\dot{H}^{\frac{1}{2}}} \neq 0.
\end{equation}
We denote this positive quantity to be $c=c(\phi)>0$. When $N$ is taken big enough, we have
\begin{equation}\label{nplus1expan}
  \|u_{N+1}(t+10^{N+1})-e^{it\Delta}u_{N+1}(10^{N+1})\|_{\dot{H}^{\frac{1}{2}}}   \geq \frac{1}{(N+1)^6}c,
\end{equation}
together with \eqref{property2} and \eqref{property5}, we have \eqref{eq2} proved noticing $N \sim log (t)$.

Now we turn to the proof of \eqref{eq1}, which is implied by \eqref{eq2}. In viewing of Strichartz estimate, $\|u(t)-e^{it\Delta}u^{+}\|_{\dot{H}_{x}^{\frac{1}{2}}}$is controlled by $\|u(s,x)\|_{L^{5}_{s,x}(t,+\infty)}$ (timing some positive constant). Thus \eqref{eq2} implies \eqref{eq1}.

The proof of Theorem \ref{mainthm} is now complete.
\section{Proof of several decay estimates}
 Now we discuss the proof of Lemma \ref{lem: decay}. We will use bootstrap argument to prove it. The proof is similar to \cite{fan2021decay} with suitable modifications (see also \cite{grillakis2013pair} and \cite{lin1978decay}). We will prove \eqref{decay1} as an example since the other two estimates are quite similar. We recall the estimate \eqref{decay1} as follows.

\begin{lemma}\label{sec3decay}
Consider $u_{\leq N}$ solves \eqref{maineq} with initial data $u_{0,\leq N}=\sum_{0 \leq n \leq N}\frac{1}{n^2}e^{-i10^{n}\Delta}\phi$. For all $t\geq 2\times 10^N$, we have 
\begin{equation}
\|u_{\leq N}(t)\|_{L^{\infty}_x} \lesssim (t-10^N)^{-\frac{3}{2}}. 
\end{equation}
\end{lemma}

\begin{remark}
One may also consider the `single bubble case' as follows, which is simpler and has a similar proof. Consider $u$ solves \eqref{maineq} with initial data $u_a(0)=e^{-ia\Delta}\phi$, 
\begin{equation}
    \|u_a(2a)\|_{L^{\infty}_x}\lesssim_{\phi} a^{-\frac{3}{2}}.
\end{equation}
\end{remark}

According to the scattering result and stability theory, for any $\epsilon>0$, fix $M>0$ such that for any $N$ large,
\begin{equation}\label{decay}
    \|u_{\leq N}(t)\|_{L^{5}_{t,x}[-10^N,-M]} \lesssim \epsilon, \quad \|u_{\leq N}(t)\|_{L^{5}_{t,x}[M,10^N]} \lesssim \epsilon,
\end{equation}
uniform in $N$.

Then we consider
\begin{equation}
A(\tau)=\sup_{-10^N \leq s\leq \tau} s^{3/2}\|u_{\leq N} (s)\|_{L_{x}^{\infty}}.
\end{equation}
Note that $A(\tau)$ is monotone increasing. We want to prove there exists $B>0$ (depends on $\phi$) such that
\begin{equation}
A(\tau) \leq B, \quad \textmd{ for } \forall \tau \geq 10^N.
\end{equation}
Recall we have Theorem \ref{global} (persistence of regularity), thus for any given $l>0$, one can find $C_{l}$ so that,
\begin{equation}
A(\tau)\leq C_{l},  \tau\leq l,
\end{equation} 
and the solution is continuous  in time in $L^{\infty}$ since we are working on high regularity data.

Thus, Lemma \ref{sec3decay} follows from the following bootstrap lemma.
\begin{lemma}\label{lem: bootstrap}
There exists a constant $C_{u_{0}}$, for that if one has $A(\tau)\leq C_{u_{0}}$, then for $\tau \geq 10^N$, one has $A(\tau)\leq \frac{C_{u_{0}}}{2}$.
\end{lemma}

Now we focus on the proof of Lemma \ref{lem: bootstrap}.
By Duhamel's Formula, we can write the nonlinear solution $u_a$ as follows,
\begin{equation}
   u_{\leq N}(t)=e^{i(t+10^N)\Delta}u_{0,\leq N}+i\int_{-10^N}^t e^{i(t-s)\Delta}(|u_{\leq N}|^2u_{\leq N})(s) ds=u_l+u_{nl}.
\end{equation}

Dispersive estimate gives for some constant $C_{0}$
\begin{equation}
\|u_{l}(t)\|_{L_{x}^{\infty}}\leq C_{0}t^{-3/2}\|\phi\|_{L_{x}^{1}}.
\end{equation} 

Now, we need an extra parameter $M$, and we split $u_{nl}$ into 
\begin{equation}
u_{nl}=F_{1}+F_{2}+F_{3}+F_{4}+F_{5},
\end{equation}
where
\begin{equation}
\begin{aligned}
    &F_1(t)=i\int_{-10^N}^{-10^N+M} e^{i(t-s)\Delta}(|u_{\leq N}|^2u_{\leq N})(s) ds,\\
    &F_2(t)=i\int_{-10^N+M}^{-M} e^{i(t-s)\Delta}(|u_{\leq N}|^2u_{\leq N})(s) ds,\\
    &F_3(t)=i\int_{-M}^M e^{i(t-s)\Delta}(|u_{\leq N}|^2u_{\leq N})(s) ds,\\
        &F_4(t)=i\int_{M}^{10^N-M} e^{i(t-s)\Delta}(|u_{\leq N}|^2u_{\leq N})(s) ds,\\
            &F_5(t)=i\int_{10^N-M}^{t} e^{i(t-s)\Delta}(|u_{\leq N}|^2u_{\leq N})(s) ds.\\
    \end{aligned}
\end{equation}
We will estimate them respectively. We note that  $t\geq 10^N$.

For $F_3$ (also $F_1$), we estimate it directly (noticing finiteness of intervals)
\begin{equation}
\begin{aligned}
\|F_{3}(t)\|_{L_{x}^{\infty}}&\leq \int_{-M}^{M}\|e^{i(t-s)\Delta}|u_{\leq N}|^{2}u(s)\|_{L_{x}^{\infty}}\\
&\lesssim Mt^{-3/2}\sup_{s}\|u_{\leq N}(s)\|_{H^{3}}^{3}\\
&\lesssim MM_{1}^{3}t^{-3/2}.
\end{aligned}
\end{equation}

For other terms we may need to use the decay property \eqref{decay}.

For $F_2$,
\begin{equation}
\begin{aligned}
\|F_{2}(t)\|_{L_{x}^{\infty}}&\leq \int_{-10^N+M}^{-M}\|e^{i(t-s)\Delta}|u_{\leq N}|^{2}u_{\leq N}(s)\|_{L_{x}^{\infty}}\\
&\lesssim t^{-3/2}\int_{-10^N+M}^{-M} \|u_{\leq N}(s)\|_{L^{3}}^{3}\\
&\lesssim t^{-3/2}\int_{-10^N+M}^{-M} \|u_{\leq N}(s)\|_{L^{2}}^{2}\|u_{\leq N}(s)\|_{L^{\infty}} \\
&\lesssim t^{-3/2}M_1^{2}\int_{-10^N+M}^{-M} |s|^{-\frac{3}{2}} \\
&\lesssim \eta M_{1}^{2}t^{-3/2},
\end{aligned}
\end{equation}
where $\eta$ is small and it depends on $M$.

For $F_4$,
\begin{equation}
\|e^{i(t-s)\Delta}|u_{\leq N}(s)|^{2}u_{\leq N}(s)\|_{L_{x}^{\infty}}\lesssim (t-s)^{-3/2}\|u_{\leq N}(s)\|_{L^{2}}^{2}\|u_{\leq N}(s)\|_{L_{x}^{\infty}}\lesssim C_{u_{0}}M_{1}^{2}(t-s)^{-3/2}s^{-3/2}.
\end{equation}
And one estimates $F_{4}$ via 
\begin{equation}
\|F_{4}(t)\|_{L^{\infty}_x} \leq CC_{u_{0}}M_{1}^{2}\int_{M}^{10^N-M}(t-s)^{-3/2}s^{-3/2}ds.
\end{equation}
We note that
\begin{equation}
    \aligned
    &\int_{M}^{10^N-M}(t-s)^{-3/2}s^{-3/2}ds\\
    &\leq \int_{M}^{\frac{t}{2}}(t-s)^{-3/2}s^{-3/2}ds\\
    &+ \int_{\frac{t}{2}}^{10^N-M}(t-s)^{-3/2}s^{-3/2}ds \\
    & \leq CM^{-\frac{1}{2}}t^{-3/2},
    \endaligned
\end{equation}
where $C$ is an universal constant (If $t\geq 2\times 10^N$, we do not need to consider the second term.). Now, we can choose $M$, such that
\begin{equation}
CM_{1}^{2}\int_{M}^{t-M}(t-s)^{-3/2}s^{-3/2}ds\leq \frac{1}{10}t^{-3/2},
\end{equation}
and we can estimate $F_{4}$ as 
\begin{equation}
\|F_{2}(t)\|_{L^{\infty}_x} \leq \frac{1}{10}C_{u_{0}}t^{-3/2}.
\end{equation}

At last, we turn to $F_5$. We recall a lemma.
 \begin{lemma}\label{lem: ele2}
 Let $f(x)$ be a $H^{4}$ function in $\mathbb{R}^{3}$, with  
 \begin{equation}
 \|f\|_{L^{2}}\leq a_{1}, \|\nabla f\|_{L_{x}^{2}}\leq a_{2}, \|f\|_{H^{4}}\leq b. 
 \end{equation}
 Then one has 
 \begin{equation}
\|f\|_{L^{\infty}}\leq  a_{1}^{2/5}a_{2}^{6/25}b^{9/25}.
 \end{equation}
 \end{lemma}

 Now the strategy is to estimate $||F_5||_{H^4}$, $||F_5||_{L_x^2}$ and $||\nabla F_5||_{L_x^2}$ respectively. Then we can apply Lemma \ref{lem: ele2}. Note that $H^{4}$ is a Banach algebra under pointwise multiplication, and $e^{i(t-s)\Delta}$ is unitary in $H^{4}$, we directly estimate $\|F_{5}(t)\|_{H^{4}}$ as 
\begin{equation}\label{eq: H4}
\|F_{3}(t)\|_{H^{4}}\leq M_{1}^{3}.
\end{equation}
Then we turn to the estimate for $||F_3||_{L_x^2}$. 
 \begin{equation}
\aligned
\big\|\int_{10^N-M}^t e^{i(t-s)\Delta}|u_{\leq N}|^2u_{\leq N}ds \big\|_{L^{2}_x} &\leq \int_{t-M}^t || |u_{\leq N}|^2u_{\leq N} ||_{L_x^2}ds \\
  &\leq  \int_{10^N-M}^t || u_{\leq N} ||^{\frac{7}{6}}_{L_x^{5}}\cdot || u_{\leq N} ||^{\frac{13}{10}}_{L_x^{\infty}} \cdot || u_{\leq N} ||^{\frac{8}{15}}_{L_x^{2}} ds \\ 
    &\leq CM_1^{\frac{8}{15}}(C_{u_{0}}t^{-\frac{3}{2}})^{\frac{13}{10}} \int_{10^N-M}^t || u_{\leq N} ||^{\frac{7}{6}}_{L_x^{5}} ds \\
     &\leq CM_1^{\frac{8}{15}}(C_{u_{0}}t^{-\frac{3}{2}})^{\frac{13}{10}} \cdot || u_{\leq N} ||^{\frac{7}{6}}_{L^{5}_t[10^N-M,t]L_x^{5}} \cdot M^{\frac{23}{30}} \\
          &\leq  CM_1^{\frac{8}{15}}M^{\frac{23}{30}}\delta^{\frac{7}{6}}(C_{u_{0}}t^{-\frac{3}{2}})^{\frac{13}{10}}. \\
\endaligned
\end{equation}
Also, we can deal with the estimate for $||\nabla F_5||_{L_x^2}$ as follows.
  \begin{equation}
\aligned
\big\|\int_{10^N-M}^t e^{i(t-s)\Delta}\nabla(|u_{\leq N}|^2u_{\leq N})ds \big\|_{L^{2}_x} &\leq \int_{10^N-M}^t || \nabla(|u_{\leq N}|^2u_{\leq N}) ||_{L_x^2}ds \\
  &\leq  \int_{10^N-M}^t || \nabla u_{\leq N} ||_{L_x^{2}}\cdot || u_{\leq N} ||^{2}_{L_x^{\infty}}  ds \\ 
    &\leq MM_1\cdot (C_{u_0}t^{-\frac{3}{2}})^2.
\endaligned
\end{equation}
At last, putting the above estimates together and applying Lemma \ref{lem: ele2}, we have,
\begin{equation}
\aligned
\big\|\int_{10^N-M}^t e^{i(t-s)\Delta}|u_{\leq N}|^2u_{\leq N}ds \big\|_{L^{\infty}_x} &\leq \big( CM_1^{\frac{8}{15}}M^{\frac{23}{30}}\delta^{\frac{7}{6}}(C_{u_{0}}t^{-\frac{3}{2}})^{\frac{13}{10}} \big)^{\frac{2}{5}} \cdot (MM_1\cdot (C_{u_0}t^{-\frac{3}{2}})^2)^{\frac{6}{25}} \cdot (M_{1}^{3}M)^{\frac{9}{25}} \\
&\leq C^{\frac{2}{5}}M_1^{\frac{23}{25}}M^{\frac{68}{75}}\delta^{\frac{7}{15}}(C_{u_0}t^{-\frac{3}{2}}). 
\endaligned
\end{equation}
Thus, by choosing $\delta$ small enough, according to $M, M_{1}$, we can ensure
\begin{equation}
\|F_{5}(t)\|_{L_{x}^{\infty}}\leq \frac{1}{10}C_{u_{0}}t^{-3/2},
\end{equation}
as desired.

To summarize, for all $t\leq \tau$, assuming $A(\tau)\leq C_{u_{0}}$, we derive
\begin{itemize}
\item For $t\leq L$, one has 
\begin{equation}
\|u_{\leq N}(t)\|_{L^{\infty}_x} \leq A(L)t^{-3/2}.
\end{equation}
\item For $L\leq t\leq \tau$, one has
\begin{equation}
 \|u_{\leq N}(t)\|_{L^{\infty}_x} \leq \{C(\|\phi\|_{L_{x}^{1}}+MM_{1}^{3})+\frac{1}{10}C_{u_{0}}+\frac{1}{10}C_{u_{0}}\}t^{-3/2}.
\end{equation}
\end{itemize}

Thus, if one choose 
\begin{equation}
C_{u_{0}}:=10A(L)+C(\|u_{0}\|_{L_{x}^{1}}+2MM_{1}^{3}),
\end{equation}
then the desired estimates follows. This ends the proof of Lemma \ref{lem: bootstrap}, which implies Lemma \ref{sec3decay}. So is Lemma \ref{lem: decay} since the other two estimates in Lemma \ref{lem: decay} can be obtained in a similar way.
\section{Proof of Theorem \ref{thm: onemore}}
Theorem \ref{thm: onemore} follows from concentration compactness argument (see \cite{kenig2006global,kenig2008global,keraani2001defect,murphy2021threshold}) and dispersive estimates.

Prove by contradiction.  We can find $\{u_{0,n}\}$ bounded in $H^{1}\cap L^{1}$, with
\begin{equation}
	\|u_{n}\|_{L_{t,x}^{5}([n,\infty)\times \mathbb{R}^3)}\geq \delta,
\end{equation}
where $u_{n}$ are the associated solutions to \eqref{maineq}.

Thus, up to picking subsequence, via by-now standard concentration compactness, we can find back ward scattering profile in the sense there is $f\in H^{1}$, with parameter $t_{n},x_{n}$, $t_{n}\rightarrow \infty$ and $x_{n}\in \mathbb{R}^{3}$, such that
\begin{equation}
\langle e^{it_{n}\Delta}u_{0,n}(x-x_{n}),f \rangle \rightarrow \|f\|_{L_x^2}^{2}>0.
\end{equation}

To see the contradiction. Since $e^{it_{n}\Delta}u_{0,n}(x-x_{n})$ are bounded in $H^{1}$, we may, up to some standard approximation arguments, assume $f$ is compactly supported. Thus $f\in L^{1}$.  Then, since $u_{0,n}$ are also bounded in $H^{1}\cap L^{1}$, we apply dispersive estimate and gives
\begin{equation}
\|e^{it_{n}\Delta}u_{0,n}(x-x_{n})\|_{L_{x}^{\infty}}\lesssim t_{n}^{-3/2}.
\end{equation}
 
 Let $n\rightarrow \infty$, and we have a contradiction.
\section{Appendix}
In the appendix, we show the scattering rate and convergence rate for NLS, which describes how fast the nonlinear solution of NLS scatters to linear solutions. We still use the 3D, cubic NLS \eqref{maineq} as the model case. Other models can be treated in a similar way. The result reads
\begin{theorem}\label{rate}
For \eqref{maineq} with initial data $u_0(x)$ in $H^4\cap L^1$, there exists constant $C_{u_0}$, for $t>0$
\begin{equation}\label{eq: rate1}
  \|u(t)-e^{it\Delta}u^{+}\|_{\dot{H}_{x}^{\frac{1}{2}}} \lesssim_{C_{u_0}} t^{-2}
\end{equation}
and
\begin{equation}\label{eq: rate2}
  \|u(t)\|_{L_{t,x}^{5}(t\geq s)}  \lesssim_{C_{u_0}} s^{-\frac{7}{10}}.
\end{equation}
\end{theorem}
\begin{remark}
For the general case, one expects: for $d$-dimensional, $H^{s_c}$-critical ($0 \leq s_c \leq 1$, corresponding nonlinear exponent $p$ satisfies $p=\frac{4}{d-2s_c}$) NLS \eqref{maingeneral} with initial data $u_0(x)$  in $H^s\cap L^1$ ($s>1$ is large enough), there exists constant $C_{u_0}$, for $t>0$
\begin{equation}
  \|u(t)-e^{it\Delta}u^{+}\|_{\dot{H}_{x}^{s_c}}  \lesssim_{C_{u_0}} t^{1-\frac{dp}{2}}
\end{equation}
and
\begin{equation}
  \|u(t)\|_{L_{t,x}^{\frac{2(d+2)}{d-2s_c}}(t\geq s)}  \lesssim_{C_{u_0}} s^{\frac{-(1+2s_c)d-2s_c}{2(d+2)}}.
\end{equation}
\end{remark}
\begin{remark}
As discussed in Introduction, based on Theorem \ref{thm: onemore}, one can improves \cite{fan2021decay} in the sense of obtaining nonlinear decay estimates with constants only depending the size of the initial data. Thus the implicit constant $C_{u_0}$ in Theorem \ref{rate} also only depends on the size of the initial data instead of the profile (shape) of the initial data.
\end{remark}
\begin{proof}
We first recall our recent result \cite{fan2021decay}. Let $u$ solves \eqref{maineq} with initial data $u_{0}$, which is in  $H^{4}\cap L^{1}$.
Then there exists a constant $C_{u_{0}}$, depending $u_{0}$, such that for $t \geq 0$,
\begin{equation}
    \|u(t,x)\|_{L_x^{\infty}}\leq C_{u_{0}} t^{-\frac{3}{2}}.
\end{equation}
By interpolation with mass, we have for $p>2$,
\begin{equation}\label{eq: ratedecay}
    \|u(t,x)\|_{L_x^{p}}\leq C_{u_{0}} t^{-3(\frac{1}{2}-\frac{1}{p})}.
\end{equation}
Then one can prove Theorem \ref{rate} by straight calculations together with Strichartz estimate and the H\"older. First, the scattering rate \eqref{eq: rate2} can be obtained directly using \eqref{eq: ratedecay}. For \eqref{eq: rate1}, recall 
\begin{equation}
    u_{+}=u_0-i\int_0^{\infty} e^{-is\Delta}|u|^{2}u(s) ds.
\end{equation}
and 
\begin{equation}
e^{-it\Delta}u(t)=u_{0}-\int_{0}^{t}e^{-is\Delta}|u|^{2}u(s) ds.
\end{equation}
Then via Strichartz estimate (using $L^1_tL^2_x$ to be the dual Strichartz norm),
\begin{align}
  \|u(t)-e^{it\Delta}u^{+}\|_{\dot{H}_{x}^{\frac{1}{2}}} &\lesssim_{C_{u_0}} \big\|\int_{0}^{t}e^{-is\Delta}(|u|^{2}u)(s) ds \big\|_{\dot{H}_{x}^{\frac{1}{2}}} \\
  &\lesssim_{C_{u_0}} \| |\nabla|^{\frac{1}{2}} |u|^{2}u(s) \|_{L^1_tL^2_x} \\
   &\lesssim_{C_{u_0}} \|u\|^2_{L^3_tL^6_x}\||\nabla|^{\frac{1}{2}}u\|_{L^3_tL^6_x} \\
    &\lesssim_{C_{u_0}} t^{-2}.
\end{align}
\end{proof}

\bibliographystyle{amsplain}
\bibliographystyle{plain}
\bibliography{BG}

\end{document}